\documentclass[12pt,a4paper]{article}

\usepackage[T2A]{fontenc}
\usepackage{indentfirst}
\usepackage{misccorr}
\usepackage{graphicx}
\usepackage{amsmath,amssymb,amsthm}

\newtheorem{theorem}{Theorem}
\newtheorem{lemma}{Lemma}
\newcommand\tig{\tilde g}

\title{On zeroes and poles of Helson zeta functions}
\author{I. Bochkov and R. Romanov\thanks{morovom@gmail.com}\\
\\
Department of Mathematics and Computer Science,\\
St Petersburg State University, Russia}
\date{}

\begin{document}

\maketitle
\begin{abstract} We show that the analytic continuations of Helson zeta functions $  \zeta_\chi (s)= \sum_1^{\infty}\chi(n)n^{-s} $ can have essentially arbitrary poles and zeroes  in the strip $ 21/40 < \Re s < 1 $ (unconditionally), and in the whole critical strip $ 1/2 < \Re s <1 $ under Riemann Hypothesis. 
\end{abstract}

Let $ \chi \colon \mathbb N \to \mathbb T $, $ \mathbb T = \{ z \in \mathbb C \colon |z|=1 \}$ be a completely multiplicative function, and let 
\begin{equation}\label{H} \zeta_\chi (s)= \sum_1^{\infty}\chi(n)n^{-s} . \end{equation}  

This series defines an analytic function in the halfplane $ \Re s > 1 $. The Euler product formula, \[ \zeta_\chi (s)=\prod_p \frac{1}{1-\chi(p)p^{-s}} , \] 
holds. In particular, the function $ \zeta_\chi $ has no zeroes for $ \Re s > 1 $. 

The study of functions $ \zeta_\chi $ was initiated by H. Helson who showed \cite{Hels} in 1969 that for almost all $ \chi$'s the function $ \zeta_\chi $ extends analytically to the halfplane $ \Re s > 1/2 $ and has no zeroes in this halfplane. Almost all here refers to the measure on the functions $ \chi $ induced by the standard product measure on the infinite dimensional torus $ \mathbb T^\infty $ via identification of $ \chi $ with the sequence $ \{ \chi (p) \} \in \mathbb T^\infty $ of its values at primes. Recently it has been shown \cite{saks} that Helson's result is optimal in that  the function $ \zeta_\chi $ does not admit meromorphic extension to the halfplane $ \Re s > \alpha $ for any $ \alpha < 1/2 $ almost surely in the above sense.

These results bring about the question of the structure of analytic continuation for $ \chi $ from the exceptional set in the Helson theorem. Such a study has been undertaken in \cite{S} where it was shown that the set of zeroes of meromorophic continuation of $ \zeta_\chi $ in the strip $ 1/2 < \Re s<1$ can be essentially arbitrary. Namely, the following theorem holds. 

\begin{theorem}{\cite[Theorem 1.4]{S}}
For any set $ \mathfrak O $ in the strip $ 1/2  < \Re s<  39/40 $ which has no accumulation points off the line $ \Re s = 1/2 $ there exists a completely multiplicative function $ \chi $ such that the Helson zeta function $ \zeta_\chi $ admits meromophic extension to the halfplane $ \Re s > 1/2 $, and 
\[ \left\{ s : \zeta_\chi ( s ) = 0,   \Re s > \frac 12 \right\} = \mathfrak O . \] 

If the RH is assumed then the same assertion holds with $ 1$ in the place of $ 39/40 $.
\end{theorem}

This theorem is proved in \cite{S} by constructing the required function $ \chi $. Given the zero set $ \mathfrak O $, the construction in \cite{S} invloves forming "dipoles" made of a zero $ \rho \in \mathfrak O $ and an additional pole nearby. The resulting function  $ \zeta_\chi $ is meromorphic, and the set of its poles is not arbitrary. A natural further question mentioned in \cite{S} is if there are any restrictions on the set of poles of $ \zeta_\chi $ in Theorem 1. Our result  is the following theorem saying essentially that the set of poles can also be arbitrary.

To account for multiple zeroes and poles we are going to use the terminology of multisets. A multiset is a pair made of a subset $ S \subset \mathbb C $ and a function $ S \longrightarrow \mathbb N $. When we say that the set of zeroes of an analytic function $ f$ coincides with a multiset, $ M_S $, we mean that $ f ( s ) = 0 $ iff $ s \in S $, and the multiplicity of zero at $ s \in S $ is $ m ( s ) $. A similar terminology is used for poles.

Our goal is to establish the following

\begin{theorem}
 1. For any disjoint multisets $  Z$ and $ P $ in the strip $ 21/40 < \Re s<1$ having no accumulation points off the line $ \Re s = 21/40  $ there exists a completely multiplicative function $ \chi $ such that $ \zeta_\chi $ admits meromorphic continuation to the halfplane $ \Re s > 21/40 $ with $ Z $ the set of its zeroes and $ P $ the set of its poles. 
 
 2. If RH holds then for any disjoint multisets $  Z$ and $  P $ in the strip $ 1/2 < \Re s<1$ having no accumulation points off the line $ \Re s = 1/2 $ there exists a completely multiplicative function $ \chi $ such that $ \zeta_\chi $ admits meromorphic continuation to the halfplane $ \Re s > 1/2 $ with $ Z $ the set of its zeroes and $  P $ the set of its poles.
\end{theorem}

\medskip 
\textit{Acknowledgements.} We are indebted to K. Seip who suggested the question about poles to us. This work was supported by Russian Science Foundation under Grant 17-11-01064.

\bigskip 

\bfseries 
\begin{center}
Logarithmic derivative
\end{center}
\mdseries

The strategy of the proof is as follows. Let $ \zeta_\chi $ be the function we seek to find. Consider $g(s)=\frac{\zeta_\chi'(s)}{\zeta_\chi (s)}$. From the Euler product representation we have 
 \begin{align*} g(s)=-\sum_{ n=1 }^\infty  \chi(n)\Lambda(n)n^{-s} = -\sum_{p,a} \chi(p^a)\Lambda(p^a)p^{-as} = \\ -\sum_p \chi(p)\Lambda(p)p^{-s} - \sum_{p, a \ge 2} \chi(p^a )\Lambda(p^a)p^{-as}, \end{align*} where $p$ ranges over the primes, $ a $ over the naturals, and $ \Lambda $ is the von Mangoldt function. The second sum in the rhs is absolutely convergent for $ \Re s > 1/2 $, hence the poles and residues of $ g $ are exactly those of the function \[ \tig (s)= -\sum_{p}  \chi(p)p^{-s} \log p . \]

Thus the problem is reduced to the one of constructing the function $ \tig $ with the required poles and residues.

\bfseries 
\begin{center}
Mellin transform	
\end{center}
\mdseries

We seek the function $\tig $ in the form
\[ \tig (s) = h(s)+\int_1^{\infty}q(x)x^{-s} dx, \Re s>1 , \] where $h$ is analytic in the halfplane $ \Re s> 1/2 $, and $ q ( x ) = o ( 1) $, $ x \to +\infty $.

The following lemma extracts from \cite[8.1]{S} a part of the argument we are going to need. A proof is provided for completeness.

\begin{lemma}
Let $ q ( x ) = o ( 1) $, $ x \to +\infty $. 
Then there exists a completely multiplicative function $\chi $ such that the function
\[  \int_1^{\infty}q(x)x^{-s} dx + \sum_{p} \chi(p)p^{-s} \log p, \]
initially defined in the halfplane $ \Re s > 1 $, extends analytically to the halfplane $ \Re s > 21/40 $ unconditionally, and to the halfplane $ \Re s > 1/2 $ if the RH holds.
\end{lemma}

\begin{proof} Arguing as in \cite[8.1]{S} we are going to construct $ \chi $ for a subset, $ \mathcal P $, of primes using the identity 
\begin{align*}
 \int_1^{\infty}q(x)x^{-s}dx +\sum_{p \in \mathcal P} \chi(p)p^{-s}\log p = \\ s\int_1^{\infty}\left( \int_1^xq(y)dy-\sum_{p\le x, p \in \mathcal P } \chi(p) \log p \right)x^{-s-1}dx . \end{align*} 
It suffices to show that there exists a $ \chi $ such that 
\begin{equation}\label{2140-e}
   r(x):= \int_1^x q(y)dy-\sum_{p\le x, p \in \mathcal P } \chi(p)\log p = O \bigl(x^{\frac{21}{40}} \log x \bigr).
\end{equation}
Let 
\begin{equation}\label{xj} x_0=2, x_{j+1}=x_j+x_j^{\frac{21}{40}}. \end{equation} 
It is enough then to establish \eqref{2140-e} at the sequence $ x= x_j$. Indeed, if it is satisfied for $ x = x_j $ then for $x \in [ x_j , x_{ j+1} ) $ we have
$$ r(x) = r(x_j) + \int_{x_j}^x q(y)dy-\sum_{x_j<p\le x, p \in \mathcal P} \chi(p)\log p .$$ The first term in rhs is $O(x^{ 21/40 } \log x )$ by assumption, \[ \int_{x_j}^xq(y)dy = O(x^{\frac{21}{40}}) \] by the boundedness of $ q$, and the rightmost term in the right hand side is trivially $ O(x^{21/40}\log x) $.
 
 It remains to choose $ \chi $ so that $ r(x_j ) = O\bigl(x_j^{21/40}\log x_j \bigr) $. In fact we are going to choose it so that $ r ( x_j ) = O ( \log x_j ) $. This is done by induction in $ j $. Let a subset $ \mathcal P \cap [ 0 , x_j ) $ and the function $ \chi $ on it  be constructed so that $ r( x_k ) \le C \log x_k $ for some $ C $ for $ k = 1, 2, \ldots, j$. Clearly
 \begin{equation} \label{esqi-e} \int_{x_j}^{x_{j+1}} q(y)dy = o\bigl( x_j^{\frac{21}{40}}\bigr). \end{equation} 
 Let $ c_j $ be the argument of the number
 $ r( x_j ) + \int_{ x_j }^{ x_{j+1}} q $, and let $ \chi (p) = e^{ i c_j } $, $ p \in \mathcal P_j $, where $ \mathcal P_j $ is a subset of primes on the interval $ [ x_j, x_{ j+1} ) $ to be chosen as follows. By construction we have 
 \[ | r ( x_{j+1} ) | \le \left| \left|r( x_j ) + \int_{ x_j }^{ x_{j+1}} q \right| - \sum_{ p \in \mathcal P_j } \log p \right| .\]
Starting from an empty set, we shall add numbers to the $ \mathcal P_j $ until the rhs becomes $ \le \log x_{ j+1} $. This is possible since 
\[ \left|r( x_j ) + \int_{ x_j }^{ x_{j+1}} q \right| = o\bigl(x_j^{\frac{21}{40}}\bigr) \] by the induction assumption and \eqref{esqi-e}, while the interval $[x_j, x_{j+1}) $ contains at least $ C x_j^{21/40}/\log x_j $ primes (\cite{BH}, p. 562) for some constant $ C>0 $ independent of $ j $ . Thus $ r_{j+1} \le C \log x_{j+1} $, as required.
 
The construction above defines $ \chi $ on a subset $ \mathcal P = \cup_j \mathcal P_j $ of primes. Extending it to all primes is done verbatim as in \cite[Lemma 5.3]{S}. This proves the unconditional part of the assertion.

Assuming the Riemann hypothesis we follow the same argument with $ x_{i+1}=x_i + 4x_i^{1/2} \log x_i $ instead of \eqref{xj}, and take into account that  the number of primes in the interval $ [ x , x + c \sqrt{x} \log x ] $ is estimated below by $ \sqrt x $ for all $ c > 3 $ \cite{Dudek}. This gives  \[  r ( x ) = O( \sqrt{x} \log^2 x ) \] which implies the required assertion in the conditional case.
 \end{proof}
 
 The function $ q$ in \cite{S} is "reverse engineered"{} from its Mellin transform, $ R $, given explicitly, and has power decay, $ q ( x) = O ( x^{ -1/40} \log^2  x$). Our construction will be done in terms of $ q $ rather than $ R $. The price to pay is that we are not going to have any control over the decay of $ q$ in the power scale. It is for this reason our stip in the unconditional result is shifted by $ 1/40 $ to the right as compared to \cite{S}. 
 
 
Thus it remains to find an analytic function $g_1 $ in the halfplane 
$ \Re s > 1 $ of the form \[ g_1(s)=\int_1^{\infty}q(x)x^{-s}dx,\] with $ q $ vanishing as $ x\to +\infty $ which admits meromorphic extension to the halfplane $\Re s>21/40 $ with the prescribed poles and respective residues. 

\begin{lemma}
Let $g$ be an analytic function in the halfplane $ \Re z > 1 $ such that 
$\sup |z|^2 |g(z)| < \infty $. Then there exists a continuous function, $q$, $ q ( s ) = o ( 1 ) $, $ s\to +\infty $, such that \[ g(s)=\int_1^{\infty}q(s)x^{-s}dx, \; \Re s > 1 . \] 
\end{lemma}

\begin{proof}
Consider the function $h(t)=g(-it+1)$. The function $h$ belongs to the Hardy class $ H^2_+ $, hence the restriction $\left. h \right|_{\mathbb R } $ is the inverse Fourier transform of a certain function $ p \in L^2 ( \mathbb R ) $, vanishing on the negative real axis. Since $ \left.h \right|_{ \mathbb R } \in L^2 \cap L^1 $, the function $ p $ is the classical Fourier transform of $ h $, hence it is continuous and vanishes as $ x \to +\infty $
by the Riemann--Lebesgue lemma. The function $q(s):= p(\log s)$ then also vanishes as $ s \to +\infty $. Finally we have (recall that $s=-it+1$):  
\begin{align*}\int_1^{\infty}q(x)x^{-s} dx =\int_0^{\infty}q(e^y)e^{(1-s)y}dy =\int_0^{\infty}p(y)e^{ity}d y=h(t)=g(s) , \end{align*} as required.
\end{proof}

\medskip
\bfseries 
\begin{center}
Construction of $ g $
\end{center}
\mdseries

Let $\alpha = 21/40 $ in the unconditional case, and $ \alpha = 1/2 $ if RH is satisifed. Assume first that the sets $ Z $ and $ P $ have no accumulation points at finite distance. In view of lemma 2, Theorem 2 for this case will be proven if we manage to find a function $g$ analytic in the halfplane $ \Re s > 1 $ and satisfying $\sup_{ \Re z > 1 } |g(z)|\, |z|^2< \infty $, which admits meromorphic extension to the halfplane $ \Re z > \alpha $ with the given poles and residues in the strip $\alpha < \Re z <1 $.

Notice first that, given a point $z_0$, $\alpha< \Re z_0< 1 $, and a number $ C > 0 $, one can choose $ n = n ( C , z_0 )$ large enough so that the function \[ g_{z_0}(z) =\frac{1}{(z-z_0)(z-z_0+1)^{2n}} \] 
has the following properties, 

(i) $|g_{z_0}(z)| <C$ for $ \Re z >1 $.

(ii) $ g_{ z_0 } $ is analytic in $\{ \Re z > \alpha \} $ except at $ z_0 $,  has a simple pole at $z_0$ with $ \operatorname{Res}_{ z_0 } g_{ z_0 } = 1 $. 

(iii) $|g_{z_0}(z)|<C$ for $|z-z_0|>20 $, $ \Re z >\alpha$. 

\begin{lemma}
Let $ \Sigma $ be a subset of the strip $\alpha< \Re z < 1 $ having no accumulation points at finite distance, and $ m \colon \Sigma \longrightarrow \mathbb Z \setminus \{ 0 \} $ be an arbitrary function. Then there exists a meromorphic function $ g $ in the halfplane $\Re z >\alpha$, whose set of poles coincides with $ \Sigma $, all poles are simple, $ \operatorname{Res}_z g =  m ( z ) $ for $ z \in \Sigma$, and \[\sup_{ \Re z > 1 } |g(z)| |z|^2 < \infty .\]
\end{lemma}

\begin{proof} Let $ G_1 $ be an arbitrary function analytic in the halfplane $\Re z>\alpha$, having no zeroes and satisfying $ G_1 ( z ) = O ( |z|^{-2} ) $ as $ |z| \to \infty $. $G_1(z)=e^{-z}z^{-2}$ will do.

Fix an arbitrary enumeration of $ \Sigma $. Given a $p_i \in \Sigma $, define $ g_i $ to be the function $ g_{ p_i } $ satisfying properties (i)--(iii) with \[ C=\frac{G_1(p_i)}{ |m (p_i)| 2^{i+1}}. \] Let
\begin{equation}\label{g} g(z)= G_1(z) \sum_i m ( p_i) \frac{g_i(z)}{G_1(p_i)}. \end{equation} Let us first check that the series in the rhs converge absolutely at any point $ z \notin \Sigma $ in the halfplane $ \Re z > \alpha $, and thus the function $ g $ is meromorphic with simple poles at $ \Sigma $ and no other singularities.

Indeed let $z \notin \Sigma$, $ \Re z > \alpha $. Clearly $|z-p_i|\le 20$ for at most finitely many $ i $'s. If $|z-p_k|>20$ for some $ k $, then \[ |g_k(z)| < \frac{G_1(p_k)}{|m ( p_k )| 2^{k+1}}, \] by (iii) from whence \[ \left| m ( p_k ) \frac{g_k(z)}{G_1(p_k)}\right|< 2^{-k-1} ,\] and the convergence is proven. The equality $ \operatorname{Res}_{ p_i } g = m ( p_i ) $ is obvious. It remains to notice that for $\Re z> 1 $ we have \[ \left| m ( p_i ) \frac{g_i(z)}{G_1(p_i)} \right|\le 2^{-i-1} ,\] hence $ | g ( z) | \le | G_1 ( z )| $, and thus $ \sup_{ \Re z \ge 1} |G_1(z)||z|^2 < \infty $, as required.
\end{proof}

Theorem 2 is thus proved in the partial case when the sets $ Z $ and $ P $ do not accumulate at finite distance. The general case is reduced to this one via dyadic decomposition of the strip in the same way as in \cite[Section 8.1]{S}.


\begin{thebibliography}{90}
\bibitem{S} K. Seip, Universality and distribution of zeros and poles of some zeta functions,  J. Anal. Math. \textbf{141}(2020), no. 1, 331--381. arXiv:1812.11729.

\bibitem{Hels} H. Helson, Compact groups and Dirichlet series, Ark. Mat. \textbf{8}(1969), 139--143.

\bibitem{BH} R. C. Baker, G. Harman, and J. Pintz, The difference between consecutive primes. II, Proc. London
Math. Soc. 83 (2001), 532--562.

\bibitem{saks} E. Saksman and C. Webb, The Riemann zeta function and Gaussian multiplicative chaos: statistics on
the critical line, arXiv:1609.00027.

\bibitem{Dudek} A. Dudek, On the Riemann hypothesis and the difference between primes, Int. J. Number Theory
\textbf{11}(2015), 771--778.

\end{thebibliography}
\end{document}